\documentclass[10pt]{amsart}

\numberwithin{equation}{section} 

\usepackage{fancyhdr}
\usepackage[stable]{footmisc}
\usepackage{bm}
\usepackage{amssymb}
\usepackage{xypic}
\usepackage{graphicx}

\newtheorem{thm}{Theorem}[section]
\newtheorem{cor}[thm]{Corollary}
\newtheorem{lem}[thm]{Lemma}
\newtheorem{prop}[thm]{Proposition}
\theoremstyle{definition}	
\newtheorem{defn}[thm]{Definition}

\def\a{\alpha}

\def\l{\lambda}

\def\vp{\varphi}

\def\cO{\mathcal{O}}

\def\DD{\mathbb{D}}
\def\FF{\mathbb{F}}
\def\GG{\mathbb{G}}

\def\N{\mathbb{N}}
\def\QQ{\mathbb{Q}}

\def\SS{\mathbb{S}}
\def\W{\mathbb{W}}
\def\Z{\mathbb{Z}}

\def\End{{{\mathrm{End}}}}
\def\Ext{\mathrm{Ext}}
\def\Gal{\mathrm{Gal}}
\def\Hom{\mathrm{Hom}}

\def\Tor{\mathrm{Tor}}

\def\holim{\mathrm{holim}}
\def\lim{\mathrm{lim}}

\def\HV1{{\FF_9[u^{\pm 1}]}}

\begin{document}

\title{A MINI-COURSE ON MORAVA STABILIZER GROUPS AND THEIR COHOMOLOGY}

\author{Hans-Werner Henn}

\address{Institut de Recherche Math\'ematique Avanc\'ee, 
C.N.R.S. - Universit\'e de Strasbourg, F-67084 Strasbourg, France}

\maketitle 


\section{Introduction}

The Morava stabilizer groups play a dominating role in chromatic stable homotopy theory. 
In fact, for suitable spectra X, for example all finite spectra,  
the chromatic homotopy type of $X$ at chromatic level $n>0$  
and a given prime $p$ is largely controlled by the continuous cohomology of a certain $p$-adic 
Lie group $\GG_n$, in stable homotopy theory known under the name of Morava stabilizer group of 
level $n$ at $p$, with coefficients in the corresponding Morava module $(E_n)_*X$.  

These notes notes are slightly edited 
notes of a mini-course of $4$ lectures delivered at the Vietnam Institute 
for Advanced Study in Mathematics in August 2013. 
The aim of the course was to introduce participants to joint work of the author 
with Goerss, Karamanov, Mahowald and Rezk which uses group cohomology in a 
crucial way to give a new approach to previous work by 
Miller, Ravenel, Wilson, and by Shimomura and his collaborators. This new approach 
has lead to a better understanding of old results as well as to substantial new results.

The notes are structured as follows. In section 2 and section 3 we give a short survey on certain 
aspects of chromatic stable homotopy theory. In section 2 we recall Bousfield localization and the 
chromatic set up. In section 3 we discuss the problem of finding finite resolutions 
of the trivial $\GG_n$-module $\Z_p$ and associated resolutions 
of the $K(n)$-local sphere and we describe known resolutions. The form of these resolutions depend on 
cohomological properties of the groups $\GG_n$ and the remaining sections concentrate on those 
properties. Section 4 contains an essentially self contained discussion of some basic group theoretical 
properties of these groups. Section 5 discusses the (co)homology of these groups with trivial 
coefficients; this is self contained except for the discussion of Poincar\'e duality and the discussion of the 
case $n=2$ and $p=3$ which is only outlined.   
Section 6 concentrates mostly on the continuous cohomology 
$H^*(\GG_1,(E_1)_*)$ and gives a fairly detailed 
account on how the short resolutions of the $\GG_1$-module $\Z_p$  
can be used to understand the homotopy of $L_{K(1)}S^0$. This homotopy is closely 
related to the image of the J-homomorphism studied in the 1960's by Adams, 
Mahowald, Quillen, Sullivan, Toda and others. Section 6 also contains some brief comments on how 
the algebraic resolutions surveyed in section 3 can be used to analyze  
$H^*(\GG_2,(E_2)_*)$, at least for odd primes. 
 
\bigskip

\section{Bousfield localization and the chromatic set up}

This section is a very brief introduction to the chromatic set up. More details with more references can be found in the introduction of \cite{GHMR}. 

\subsection{Bousfield localization}

Let $E_*$ be a generalized homology theory. Bousfield locali{\-}zation 
with respect to $E_*$ is a functor $L_E$ from spectra to spectra 
together with a natural transformation $\lambda:X\to L_EX$ which is 
{\it terminal} among all $E_*$-equivalences. $L_E$ exists for all homology 
theories $E_*$ \cite{Bo}. Bousfield-localization makes precise the 
idea to ignore spectra which are trivial to the eyes of $E_*$-homology.  

\noindent
\underbar{Example} Let $MG$ be a Moore spectrum for an abelian group $G$. 
Then $L_{M\Z_{(p)}}$ resp. $L_{M\QQ}$ are the 
homotopy theoretic versions of arithmetic lo{\-}calization 
with respect to $\Z_{(p)}$ resp. $\QQ$ (e.g. homology groups and homotopy groups of a
spectrum get localized by these functors).

\subsection{Morava $K$-theories} Fix a prime $p$. We are interested in 
the localization functors $L_{K(n)}$ with respect to Morava 
$K$-theory $K(n)$. We recall that $K(n)$ is a multiplicative 
periodic cohomology theory 
with coefficient ring $K(n)_*=\FF_{p}[v_n^{\pm 1}]$, where 
$v_n$ is of degree $2(p^n-1)$ if $n>0$. In case $n=0$ the convention is that 
$K(0)=M\QQ$, independant of $p$. Furthermore $K(n)$ admits a theory 
of characteristic classes and the associated 
formal group law $\Gamma_n$ is the Honda formal group law of height $n$. 

The functors $L_{K(n)}$ are elementary 
``building blocks'' of the stable homotopy category of finite 
$p$-local complexes in the following sense.  

a) The localization functor $L_{K(n)}$ is ``simple'' in the sense 
that the category of $K(n)$-local spectra 
contains no nontrivial localizing subcategory, i.e. no non-trivial thick 
subcategory which is closed under arbitrary coproducts \cite{HS}     

b) There is a tower of localization functors 
$$
\ldots \to L_n\to L_{n-1}\to \ldots
$$ 
(with $L_n=L_{K(0)\vee\ldots \vee K(n)}$) 
together with natural transformations $id\to L_n$ such that 
$$
X\simeq \holim_n L_nX
$$ 
for every finite $p$-local spectrum $X$. Furthermore, for each $n$ and $p$  
there is a homotopy pullback diagram (a ``chromatic square'')
$$
\xymatrix@C=15pt{
L_nX \rto  \dto & L_{K(n)}X \dto \\  
L_{n-1}X \rto& L_{n-1}L_{K(n)}X \\
}
$$ 
i.e. $L_n$ is determined by $L_{K(n)}$ and $L_{n-1}$.

The functors $L_{K(n)}$ do not commute with smash products. Therefore the appropriate smash 
product of $K(n)$-local spectra $X$ and $Y$ is given by 
$X\wedge_{K(n)}Y:=L_{K(n)}(X\wedge Y)$.

\subsection{$L_{K(n)}S^0$ as homotopy fixed point spectrum}
\medskip

The functors $L_{K(n)}$ are controlled by cohomological properties of the 
Morava stabilizer group $\SS_n$ resp. $\GG_n$  where $\SS_n$ is the group 
of automorphisms of the formal group law $\Gamma_n$ 
(extended to the finite field $\FF_{q}$ with $q=p^n$). The Galois group $\Gal(\FF_q:\FF_p)$ 
acts on $\SS_n$ and $\GG_n$ is defined as semidirect product 
$\GG_n=\SS_n\rtimes \Gal(\FF_q:\FF_p)$. This group  
acts on the Lubin-Tate ring which classifies deformations  of $\Gamma_n$ (in the sense of Lubin-Tate). 
The Lubin-Tate spectrum $E_n$ is a complex oriented $2$-periodic cohomology theory  
whose associated formal group law is a universal deformation of $\Gamma_n$; 
its homotopy groups are given as $(E_n)_*=\pi_*(E_n)=\pi_0(E_n)[u^{\pm 1}]$ with $u\in \pi_{-2}(E)$ 
and $\pi_0(E_n)\cong \W_{\FF_{q}}[[u_1,\ldots,u_{n-1}]]$, the ring of power series on $n-1$ generators 
over the ring of Witt vectors of $\FF_q$. The group $\GG_n$ acts on 
deformations and hence on $(E_n)_*$, and by the Hopkins-Miller-Goerss theorem \cite{GoH} this action 
can be lifted to $E_{\infty}$-ring spectra, i.e. $\GG_n$ acts on $E_n$ through $E_{\infty}$-maps.  
 
By Devinatz-Hopkins \cite{DH} the ``homotopy fixed point spectrum''
$E_n^{h\GG_n}$ can be identified with $L_{K(n)}S^0$ and its 
Adams-Novikov spectral sequence can be identified with  the associated 
homotopy fixed point spectral sequence 
\begin{equation}\label{descentSS}
E_2^{s,t}\cong H^s_{cts}(\GG_n,(E_n)_t)\Longrightarrow 
\pi_{t-s}L_{K(n)}S^0\ . 
\end{equation}
Therefore methods of group theory and group cohomology can be used to study the $K(n)$-local 
sphere and more generally the $K(n)$-local category.

\noindent
\underbar{Warning:} The ``homotopy fixed point spectrum'' is taken with 
respect to the action of a profinite group. We will not try to explain how 
this is done in detail but we insist that in \cite{DH} there is a construction such that there is an 
associated homotopy fixed point spectral sequence with an $E_2$-term 
which is given in terms of continuous group cohomology as in (\ref{descentSS}).

\section{Resolutions of $K(n)$-local spheres}

The case $n=0$ is both exceptional and trivial:  $K(0)=M\QQ=H\QQ$ (with $H\QQ$ 
the Eilenberg-MacLane spectrum for the rationals) and $L_{K(0)}$ is rationalization. 
From now on we will assume $n>1$. 

\subsection{The example $n=1$ and $p>2$} 

The case $n=1$ is well understood. In this case we have $E_1=K\Z_p$ 
($p$-adic complex $K$-theory). The formal group law $\Gamma$ is the multiplicative 
group law given by $1+(x+_{\Gamma}y)=(1+x)(1+y)$. The endomorphism ring of $\Gamma$ over 
$\FF_p$ is isomorphic to $\Z_p$: in fact, the element $p\in \Z_p$ corresponds 
to the endomorphism $[p]_{\Gamma}(x)=(1+x)^p-1\equiv x^p\mod(p)$ and the 
canonical homomorphism $\Z\to \End(\Gamma), n\mapsto [n](x)$ extends to an 
continuous isomorphism $\Z_p\to \End(\Gamma)$.  Therefore the group 
$\GG_1=\SS_1$ can be identified with $\Z_p^{\times}$, 
the units in the $p$-adic integers. The group acts on $K\Z_p$ by Adams 
operations, and the action on its homotopy $\pi_*(K)=\Z_p[u^{\pm 1}]$ is via graded 
ring automorphisms determined by $(\l,u)\mapsto \l u$. 
\smallskip
If $p$ is odd then 
$\Z_p^{\times}\cong C_{p-1}\times \Z_p$,  
and the homotopy fixed points with respect to $\Z_p^{\times}$ can be formed in two steps, 
first with respect to the cyclic group $C_{p-1}$ and then with respect to $\Z_p$. 
Taking homotopy fixed points with respect to $C_{p-1}$ is quite simple; 
on homotopy groups it amounts to taking invariants with respect 
to the action of $C_{p-1}$. Hence we get  
$$
\pi_*(K\Z_p^{hC_{p-1}})\cong \Z_p[u^{\pm (p-1)}]\ . 
$$ 
In fact, $K\Z_p^{hC_{p-1}}$ is the Adams summand of 
$K\Z_p$. The Adams operation $\psi^{p+1}$ still acts on $K\Z_p^{hC_{p-1}}$,  
taking homotopy fixed points with respect to $\Z_p$ amounts to taking 
the fibre of $\psi^{p+1}-id$ and we get a fibration  
\begin{equation}\label{LK1odd}
L_{K(1)}S^0\to K\Z_p^{hC_{p-1}}\buildrel{\psi^{p+1}-id}\over\longrightarrow 
K\Z_p^{hC_{p-1}}\ . 
\end{equation} 

\noindent
We will get back to this in section \ref{p>2}.

\subsection{The case that $p-1$ does not divide $n$}\label{p-1not}
\medskip

The fibration (\ref{LK1odd}) can be considered as an example of a $K\Z_p$-resolution 
in the sense of Miller \cite{MillerJPAA}.  

Following Miller we say that a $K(n)$-local spectrum $I$ is {\it $E_n$-injective} if  
the canonical map  $I\to L_{K(n)}(E_n\wedge I)$ splits, i.e. it has a left inverse in the homotopy category. 
A sequence of maps $X_1\to X_2\to X_3$ is said to be {\it $E_n$-exact} if the composition 
of the two maps is nullhomotopic and if $[-, I]_*$ transforms  $X_1\to X_2\to X_3$ 
into an exact sequence of abelian groups for each $E_n$-injective 
spectrum $I$. An {\it $E_n$-resolution} of a spectrum $X$ is a sequence 
$$
I_{\bullet}: *\to X\to I^0\to I^1\to \ldots
$$
such that the sequence is $E_n$-exact and each $I^s$ is $E_n$-injective. If there exists an integer 
$k\geq 0$ such that $k$ is minimal with the property that $I^s$ is contractible for all $s>k$ then we say 
that the $E_n$-resolution is of length $k$. 

The spectrum $E_n$ is $E_n$-injective because $E_n$ is $K(n)$-local and a ring spectrum.

The following result is in essence due to Morava.
 
\begin{thm}\label{En-res}\cite{HennRes} 
If $n$ is neither divisible by $p-1$ nor by $p$ 
then $L_{K(n)}S^0$ admits an $E_n$-resolution of length $n^2$ 
in which each $I^s$ is a summand in a finite wedge of $E_n$'s. 
\end{thm}
\smallskip

\noindent
\underbar{Remarks} a) Suppose $G=\lim_{\a}G_{\a}$ is a profinite group and suppose $p$ is a prime.  
We write $\Z_p[[G]]=\lim_{\a}\Z_p[G_{\a}]$ for the profinitely completed group algebra over $\Z_p$. 
Likewise, for a profinite set $S=\lim_{\a}S_{\a}$ we write $\Z_p[[S]]$ for $\lim_{\a}\Z_p[S_{\a}]$.  
The theorem is derived from the existence of a finite projective resolution of length $n^2$    
$$
P_{\bullet}:   0\to P_{n^2}\to \ldots\to P_0\to \Z_p\to 0
$$ 
of the trivial $\GG_n$-module $\Z_p$  in the category of 
profinite $\Z_p[[\GG_n]]$-modules. 
A more precise form of the theorem is that the $E_n$-resolution ``realizes the projective resolution" in the sense that there is an isomorphism of chain complexes 
\begin{equation}\label{realisation}
\Hom_{cts}(P_{\bullet},(E_n)_*)\cong E_*(I_{\bullet}) 
\end{equation} 
where here and elesewhere in these notes we adopt the convention that 
$(E_n)_*X$ for $K(n)$-local $X$ means $\pi_*(E_n\wedge_{K(n)}X)$. 

b) The assumption that $n$ is divisible by $p$ (but not by $p-1$) is not 
a very serious restriction. There is still a useful variation of this theorem which holds. 
However, the assumption that $n$ is not divisible by $p-1$ is quite crucial.

\subsection{The example $n=2$ and $p>3$}\label{n=2-p>3}

In the case $n=2$ and $p>3$ (even if $p=3$) the group $\GG_2$ can be decomposed 
as a product $\GG_2\cong \GG_2^1\times \Z_p$ (cf. section \ref{reduced}). 
The following two results are analogues of results of Ravenel 
(cf. chapter 6 of \cite{Ravgreen}).

\begin{thm}\label{G2res}\cite{HennRes}  
There is an exact complex of projective $\Z_p[[\GG_2^1]]$-modules 
$$ 
0\to C_3\to C_2\to C_1\to C_0\to \Z_p\to 0 
$$ 
with 
$C_0=C_3=\Z_3[[\GG_2/F_{2(p^2-1)}]]$ and 
$C_1=C_2=\Z_3[[\GG_2^1]]\otimes_{\Z_p[F_{2(p^2-1)}]}\l_{1-p}$ 
where $\l_{1-p}$ is a certain projective $\Z_p[F_{2(p^2-1)}]$-module of $\Z_p$-rank $2$ 
and $F_{2(p^2-1)}$ is a maximal finite subgroup of order $2(p^2-1)$ of $\GG_n^1$. 
\end{thm}

\begin{thm}\label{K2res}\cite{HennRes} There exists 
a fibration 
$$
L_{K(2)}S^0\to E_2^{h\GG_2^1}\longrightarrow 
E_2^{h\GG_2^1}
$$
and an $E_2$-resolution
$$
*\to E_2^{h\GG_2^1}\to X_0\to X_1\to X_2\to X_3\to *
$$ 
with $X_0=X_3=E_2^{hF_{2(p^2-1)}}$ and $X_1=X_2=\Sigma^{2(p-1)}E_2^{hF_{2(p^2-1)}}\vee 
\Sigma^{2(1-p)}E_2^{hF_{2(p^2-1)}}$.  
\end{thm}

\subsection{The example $n=1$ and $p=2$} 
\medskip

This case is again well understood. 
The isomorphism  $\GG_1=\Z_2^{\times}\cong C_2\times\Z_2$ allows, 
as before, to form the 
homotopy fixed points in two stages and we obtain the following fibration  
\begin{equation}\label{K1resp=2}
I_{\bullet}: L_{K(1)}S^0\to K\Z_2^{hC_{2}} 
\buildrel{\psi^{3}-id}\over\longrightarrow K\Z_2^{hC_{2}} \ . 
\end{equation}

The homotopy fixed points $K\Z_2^{hC_2}$ can be identified with $2$-adic 
real $K$-theory $KO\Z_2$. 
Note that this is not an example of Theorem \ref{En-res}, in fact a finite length 
$E_n$-resolution as in Theorem \ref{En-res} cannot exist in this case 
because $\GG_1$ contains an element of order $2$, and hence $H^*_{cts}(\GG_1,\FF_2)$ 
is nontrivial in arbitrarliy high cohomological degrees. 
Nevertheless this is a very useful substitute. We will get back to this in section \ref{p=2}.

\subsection{The general case $p-1$ divides $n$} 
\medskip

The natural question arises whether there are generalizations 
of the fibre sequence (\ref{K1resp=2}) for higher $n$ and $p$ such that $p-1$ divides $n$. 
What could they look like? 
In other words, can we explain the appearance of $K\Z_2^{hC_2}$ in (\ref{K1resp=2}) 
so that it fits into a more general framework? 

A good point of view is provided by {\it group cohomology} as follows:

Applying the functor ${K\Z_2}_*$ to (\ref{K1resp=2}) gives a short exact sequence 
$$
0\to {K\Z_2}_*\to {K\Z_2}_*(K\Z_2^{hC_2})\to {K\Z_2}_*(K\Z_2^{hC_2})\to 0 
$$ 
in which  ${K\Z_2}_*(K\Z_2^{hC2})$ can be identified with the group of continuous homomorphisms 
from the permutation module $\Z_2[[\Z_2^{\times}/C_2]]$ to $(K\Z_2)_*$. 
The fibre sequence (\ref{K1resp=2}) can therefore be considered as a homotopy theoretic realization  
of the exact sequence of profinite $\Z_2[[\Z_2^{\times}]]$-modules (cf. section \ref{n=1})
\begin{equation}\label{Z2-res}
P_{\bullet}: 0\to\Z_2[[\Z_2^{\times}/C_2]]\to \Z_2[[\Z_2^{\times}/C_2]]\to \Z_2\to 0 \ . 
\end{equation}
in the sense that ${K\Z_2}_*(I_{\bullet})\cong \Hom_{cts}(P_{\bullet},{K\Z_2}_*)$ where 
$I_{\bullet}$ is the fibration of (\ref{K1resp=2}) and $P_{\bullet}$ the exact sequence of 
(\ref{Z2-res}). However, in this case $I_{\bullet}$ is not a $K\Z_2$-resolution 
in the sense of section \ref{p-1not} and $P_{\bullet}$ is not  a free (neither a projective) resolution 
but rather a  resolution by permutation modules. 
\medskip

This suggests that we should look for 
a resolution of the trivial $\GG_n$-module $\Z_p$ in terms of permutation 
modules $\Z_p[[\GG_n/F]]$ with $F$ running through finite subgroups 
(or summands thereof) and try to realize those in the sense of (\ref{realisation}). 
In fact, if $F$ is any finite subgroup of $\GG_n$ 
there is a canonical isomorphism 
\begin{equation}\label{EnEnhF}
(E_n)_*E_n^{hF}\cong \Hom_{cts}(\Z_p[[\GG_n/F]],{E_n}_*)\ .
\end{equation}
\smallskip
This leads to the following questions? 
\smallskip 

\noindent
\underbar{Questions:} 
1) Are there resolutions of finite length and finite type of the trivial 
$\Z_p[[\GG_n]]$-module $\Z_p$ by (direct summands of) 
permutation modules of the form $\Z_p[[\GG_n/F]]$ for finite subgroups  $F\subset\GG_n$? 

\noindent
2) Can these resolutions be realized by resolutions
of spectra where the resolving spectra 
are the corresponding homotopy fixed point spectra 
with respect to these finite subgroups?

\noindent 
3) If the answers to (a) and (b) are yes, how unique are these resolutions?  

Here we call a sequence of spectra
$$
*\to X=X_{-1}\to X_0\to X_1\to \ldots 
$$ 
a {\it resolution} of $X$ if the composite of any two consecutive maps is nullhomotopic and if each of the 
maps $X_i\to X_{i+1}$, $i\geq 0$, can be factored as $X_i\to C_i\to X_{i+1}$ such that
$C_{i-1}\to X_i\to C_i$ is a cofibration for every $i\geq 0$ (with $C_{-1}:=X_{-1}$). 
We say that the resolution is of length n if $C_n\simeq X_n$ and $X_i\simeq *$ if $i>n$.

\noindent
\underbar{Remark} The group $\SS_n$ is of finite virtual mod-$p$ cohomological 
dimension ($vcd_p$) equal to $n^2$, i.e. there is a finite index subgroup whose 
continuous mod-$p$ cohomology vanishes in degrees $>n^2$. In the 
case of a discrete group $G$ of finite $vcd_p$ there is a geometric source for 
resolutions of the trivial module $\Z_p$ 
by permutation modules of the form $\Z_p[G/F]$ with $F$ runing through finite 
subgroups. In fact, they can be obtained as the cellular chains of a contractible finite dimensional 
$G$-$CW$-complex on which $G$ acts with finite stabilizers. 
Such spaces always exist (if $vcd_p(G)<\infty$) and 
hence such resolutions always exist.  In our case such spaces are not known to exist 
and we have to manufactor our resolutions by hand.

\subsection{The example $n=2$ and $p=3$}\label{n=2-p=3}
This is the first new case.

\begin{thm}\label{algduality}\cite{GHMR} 
There is an exact complex of $\Z_3[[\GG_2^1]]$-modules 
$$ 
0\to C_3\to C_2\to C_1\to C_1\to \Z_3\to 0  
$$ 
with 
$C_0=C_3=\Z_3[[\GG_2/G_{24}]]$ and $C_1=C_2=\Z_3[[\GG_2^1]]\otimes_{\Z_3[SD_{16}]}\chi$ 
where $SD_{16}$ is a maximal finite subgroup of $\GG_2$ which is isomorphic to the semidihedral 
group of order $16$, $\chi$ is a suitable character $\chi$ of $SD_{16}$ defined over $\Z_3$, 
and $G_{24}$ is another maximal finite subgroup of order $24$ of $\GG_2$.  
\end{thm}
\smallskip 

\begin{thm}\cite{GHMR}\label{GHMR}  There exists 
a fibration 
$$
L_{K(2)}S^0\to E_2^{h\GG_2^1}\longrightarrow 
E_2^{h\GG_2^1}
$$
and a resolution of $E_2^{h\GG_2^1}$  of length $3$ 
$$
*\to E_2^{h\GG_2^1}\to X_0\to X_1\to X_2\to X_3\to * 
$$
with $X_0=E_2^{hG_{24}}$, $X_1=\Sigma^8E_2^{hSD_{16}}\simeq X_2=\Sigma^{40}E_2^{hSD_{16}}$ 
and $X_3=\Sigma^{48}E_2^{hG_{24}}$. 
\end{thm}

\noindent
\underbar{Remarks} a) The homotopy fixed point spectrum $E_2^{hSD_{16}}$ is $16$-periodic and the 
suspensions $\Sigma^8E_2^{hSD_{16}}$ and $\Sigma^{40}E_2^{hSD_{16}}$ 
are due to the presence of the character $\chi$ in the previous theorem. The $(E_2)_*$-homology of 
$E_2^{hG_{24}}$ is $24$-periodic and this resolution realizes the one of the previous theorem 
in the same sense as before, i.e. there is an isomorphism of complexes 
$(E_2)_*(X_{\bullet})\cong \Hom_{cts}(C_{\bullet},(E_2)_*)$. However, the spectrum $E_2^{hG_{24}}$ 
itself is only $72$-periodic and the $48$-fold suspension appearing with $X_3$ is a homotopy theoretic 
subtlety which is not explained by the algebra. 

b) The spectrum $E_2^{hG_{24}}$ is a version of the Hopkins-Miller 
higher real $K$-theory spectrum $EO_2$. It is equivalent to $L_{K(2)}tmf$, the 
$K(2)$-localization of the spectrum $tmf$ of topological modular forms at $p=3$.

There is a second resolution which can be described as follows: 
we choose an $8$-th primitive roof of unity in $\W_{\FF_9}$. This defines a one-dimensional faithful  
representation of $C_{8}$ over $\W_{\FF_9}$ which we denote it by $\l_1$, and its $k$-th 
tensor power by $\l_k$. Then the $\l_k$ are naturally $\Z_3[SD_{16}]$-modules 
and $\l_4$ splits as $\l_{4,+}\oplus \l_{4,-}$. 
Furthermore $\l_{4,-}$ is the representation $\chi$ of \ref{algduality}. 
\medskip

The following results are implicit in \cite{GHM}. 

\begin{thm}\cite{HennRes}\label{CentHenn} 
There is an exact complex of $\Z_3[[\GG_2^1]]$-modules
\begin{align*}
0\to \Z_3[[\GG_2^1/SD_{16}]] & 
\to \Z_3[[\GG_2^1]]\otimes_{\Z_3[SD_{16}]}\l_2\to \\ 
(\Z_3[[\GG_2^1]] \otimes_{\Z_3[G_{24}]}&\widetilde{\chi})\oplus 
(\Z_3[[\GG_2^1]]\otimes_{\Z_3[SD_{16}]}\l_{4,-})  
\to \Z_3[[\GG_2^1/G_{24}]]\to \Z_3\to 0   
\end{align*}
where $\widetilde{\chi}$ is a suitable nontrivial one-dimensional character of $G_{24}$ 
defined over $\Z_3$.  
\end{thm} 

\begin{thm}\cite{HennRes}\label{HennRes}  There exists a resolution of $E_2^{h\GG_2^1}$ of length $3$  
\begin{align*}
*\to E_2^{h\GG_2^1}\to E_2^{hG_{24}}\to & 
\Sigma^{36}E_2^{hG_{24}}\vee\Sigma^{8}E_2^{hSD_{16}}\to \\
\to &  \Sigma^{4}E_2^{hSD_{16}}\vee\Sigma^{12}E_2^{hSD_{16}}\to 
E_2^{hSD_{16}}\to * \ . 
\end{align*}
\end{thm}

\noindent
\underbar{Remark} As in Theorem \ref{GHMR} the suspensions are due 
to the presence of the characters in the previous theorem.

\subsection{Permutation resolutions  and realizations}{\ }
\medskip

\begin{prop}\cite{HennRes}  Let $p$ be an odd prime and $n=k(p-1)$ with 
$k\not\equiv 0\ \text{mod}\ p$. 
Then the trivial $\Z_p[[\GG_n]]$-module $\Z_p$ 
admits a resolution of finite length in which all modules are 
finite direct sums of modules which are of the form $\Z_p[[\GG_n/F]]$ 
with $F$ a finite subgroup of $\GG_n$.  
\end{prop}

In the case of general profinite groups $G$ work of Symonds \cite{sy} suggests 
that such resolutions exist under suitable finiteness assumptions on $G$. 
In the case of the stabilizer group \cite{HennRes} provides a more direct approach to their construction.   

\medskip

\begin{thm}\cite{HennRes}  For $p$ odd and $n=p-1$ there is a resolution 
of $L_{K(n)}S^0$ of finite length in which all spectra are summands in 
finite wedges of spectra of the form $E_n^{hF}$ and $F$ is a finite 
subgroup of $\GG_n$.  
\end{thm}

\subsection{Applications and work in progress}\label{progress} 

\smallskip 
The pioneering work of Shimomura and collaborators  on calculating 
the homotopy groups $\pi_*(L_{K(2)}X)$ for $X=S^0$ \cite{sh-w} resp. the Moore spectrum $V(0)$ 
\cite{shi1} at the prime $3$ and of $\pi_*(L_{K(2)}S^0)$ for primes $p>3$ \cite{sh-y} 
have been poorly understood by the community. Therefore an alternative approach 
(using group cohomology in a systematic way) is useful. 
Accomplished respectively ongoing projects include the following: 

\noindent
\subsubsection{} \label{p=3} The exact complex of Theorem \ref{algduality} has been made into an efficient calculational tool in 
the thesis of Nasko Karamanov \cite{kar}. This has lead to calculations at $p=3$ 
of $\pi_*(L_{K(2)}X)$ for $X=V(1)$, the cofibre of the Adams self map of $V(0)$ \cite{GHM},  
as well as for $V(0)$ \cite{HKM}. The results in \cite{HKM} refine Shimomura's results 
of \cite{shi1} and correct some errors. The case of $S^0$ is a joint project with 
Goerss, Karamanov and Mahowald. Details should appear in the near future.

The main result of \cite{GHMR} together with partial information from \cite{HKM} have lead to 
major structural results on the homotopy category of $K(2)$-local spectra at the prime $3$: 
the rational homotopy of $L_{K(2)}S^0$ has been calculated and the chromatic splitting conjecture for 
$n=2$ and $p=3$ has been established in \cite{GHM2}, 
the Picard group of smash-invertible $K(2)$-local spectra has been calculated in \cite{kar1} and 
\cite{Picat3} and the Brown-Comenetz dual of the sphere has been determined in \cite{BCdual}.   

\noindent
\subsubsection{}\label{p>3} 
The exact complex of Theorem \ref{G2res} has been turned into an efficient calculational tool 
in the thesis  of O. Lader \cite{Ld}. Among other things he has recovered Shimomura's calculation of 
$\pi_*L_{K(2)}V(0)$ and Hopkins unpublished calculation of the Picard group $Pic_2$, both 
for primes $p>3$. 

\noindent
\subsubsection{}\label{p=2a} 
Resolutions for $n=p=2$ which resemble those of section \ref{GHMR}
were announced in \cite{HennRes} although the precise form of $X_3$ in the analogue 
of Theorem \ref{GHMR} remained unclear at the time. These resolutions 
have since been constructed in the recent Northwestern theses of 
Agn\`es Beaudry and Irina Bobkova. Beaudry has used this to disprove the chromatic splitting 
conjecture at $n=p=2$ \cite{Be}.  The resolutions can be expected to lead to further progress in 
$K(2)$-local homotopy at the prime $2$ similar to the case of the prime $3$ mentioned in subsection 
\ref{p=3} above.  

\smallskip 

In the remaining sections 4-6 of these notes we will explain some of the algebraic aspects of this story 
in more detail, in particular group theoretical and cohomological properties of $\GG_n$.  
The homotopy theoretic aspects will mostly remain in the background. 
\medskip

\section{The Morava stabilizer groups. First properties}

There are different ways to discuss these groups. They arise in stable 
homotopy theory as automorphism groups of certain $p$-typical formal group laws
$\Gamma_n$ defined over $\FF_p$. For our purposes it seems best to introduce them as follows. 
\medskip

\begin{defn} Let $p$ be a prime and let $\cO_n$ be the non-commutative algebra over 
$\W(\FF_{p^n})$, the ring of Witt vectors for the field $\FF_{p^n}$, 
generated by an element $S$ subject to the relations 
$S^n=p$ and $Sw=w^{\sigma}S$ for each $w\in \W(\FF_{p^n})$ where $w^{\sigma}$ is 
the result of applying the lift of Frobenius on $w$. In other words  
\begin{equation}\label{def-On}
\cO_n=\W(\FF_{p^n})\langle S\rangle/(S^n=p, Sw=w^{\sigma}S)\ . 
\end{equation}
\end{defn}

\noindent
\underbar{Remarks} (on Witt vectors) a) The ring of Witt vectors $\W(\FF_{p^n})$ is a 
$\Z_p$-algebra which is a complete local ring with maximal ideal $(p)$.  
It is an integral domain which is free of rank $n$ as $\Z_p$-module. As the notation suggests $\W$ 
is a functor, say from the category of finite field extensions  of $\FF_p$ to the category of 
integral domains which are unramified $\Z_p$-algebras. 

\noindent
b) Because of functoriality the Frobenius automorphism of $\FF_{p^n}$ lifts to a 
$\Z_p$-algebra automorphism. 

\noindent
c) By Hensel's lemma each root of unity in $\FF_{p^n}^{\times}$ lifts uniquely to a root of unity in 
$\W(\FF_{p^n})$.  

\noindent
d) Each element of $w\in \W(\FF_{p^n})$ can be uniquely written as $\sum_{i\geq 0}w_ip^i$ where 
all $w_i\in \W(\FF_{p^n})$ satisfy $w_i^{p^n}=w_i$. 
(Already for $n=1$ this is a non-trivial statement). 

\noindent
e)  A concrete construction (which, however, does not immediately reveal the functoriality 
of the construction) can be given as follows. Over $\FF_p[X]$ the polynomial $X^{p^n}-X$ 
can be factored as product of irreducible polynomials whose degrees divide $n$. 
For each divisor $d$ of $n$ 
there is at least one factor $p_d$ of degree $d$. Then $\FF_{p^n}\cong \FF_p[X]/(p_n)$ 
and $\W(\FF_{p^n})\cong\Z_p[X]/(\widetilde{p}_n)$ where $\widetilde{p}_n$ is any lift of $p_n$ 
to a polynomial $\widetilde{p}_n\in \Z_p[X]$.  

\noindent
\underbar{Remarks} (on $\cO_n$) a) The left $\W(\FF_{p^n})$-submodule of $\cO_n$ generated by 
$S$ is a two sided ideal with quotient $\cO_n/(S)\cong \FF_{p^n}$ and $\cO_n$ is 
complete with respect to the filtration given by the powers of the ideal $(S)$. 
In fact, $\cO_n$ is a non-commutative complete discrete valuation ring. 
The valuation $v$ is normalized such that $v(p)=1$, i.e. $v(S)=\frac{1}{n}$. 

\noindent
b) $\cO_n$ is a free $\W(\FF_{p^n})$-module of rank $n$. 
A basis is given by the elements $1,S,\ldots S^{n-1}$ and every element in $x\in \cO_n$ can be 
uniquely written as 
$$
x=\sum_{i=0}^{n-1}a_iS^i
$$ 
with $a_i\in \W(\FF_{p^n})$, and thus as 
$$
x=\sum_{j=0}^{\infty}x_jS^j
$$ 
with all $x_j\in \W(\FF_{p^n})$ satisfying $x_j^{p^n}=x_j$. In fact, if $a_i=\sum_{j=0}^{\infty}a_{i,j}p^j$ 
then $x_{i+jn}=a_{i,j}$.   

\noindent 
c) Inverting $p$ makes $\cO_n$ into a division algebra $\DD_n$ 
which is central over $\QQ_p$ and free of rank $n^2$ as a vector space over $\QQ_p$. 
In fact, $\cO_n$ is a domain and if $x=\sum_{j\geq k}{x_j}S^j$ 
with $x_k\neq 0$, then $x=S^kx'$ and $x'$ is invertible in $\cO_n$. 
Inverting $p$ also inverts $S$ and thus every nontrivial element admits an inverse. 

\noindent
d) The Galois group $Gal(\FF_p^{n}:\FF_p)$ of the extension $\FF_p\subset\FF_{p^n}$ 
acts on $\cO_n$ by algebra automorphisms 
via $(\sigma,\sum_{i=0}^{n-1}x_iS^i)\mapsto \sum_{i=0}^{n-1}x_j^{\sigma}S^j$ where 
as before $x_j^{\sigma}$ is the result of applying the lift of Frobenius to $x_j$. 
We note that by the relation in (\ref{def-On}) this action of Frobenius can be realized by conjugation 
by $S$ inside $\DD_n^{\times}$. 

\medskip

\begin{defn}\label{stab} The {\it $n$-th Morava stabilizer group} at $p$ is defined as 
the group of units in $\cO_n$. It is denoted $\SS_n$, i.e. $\SS_n=\cO_n^{\times}$. 
The {\it extended $n$-th Morava stabilizer  group at $p$} is the semidirect product 
$\GG_n:=\SS_n\rtimes Gal(\FF_p^{n}:\FF_p)$. 
\end{defn} 
\smallskip

\noindent
\underbar{Remarks} a) Because $\cO_n$ is a complete (non-commutative) discrete valuation ring, an 
element $x\in \cO_n$ is invertible in $\cO_n$ if and only if $v(x)=0$.

\noindent
b) It can be shown that $\SS_n$ is the group of automorphisms of 
a suitable formal group law $\Gamma_n$ (associated to the complex oriented cohomology theory 
given by Morava $K$-theory $K(n)$). The group law $\Gamma_n$ is already defined over $\FF_p$ 
but $\SS_n$ is its automorphism group considered as a formal group law over the field $\FF_{p^n}$.

\subsection{The Morava stabilizer group as a profinite group}

The filtration of $\cO_n$ by powers of $(S)$ leads to a very useful filtration of $\SS_n$. 
For $i=\frac{k}{n}$ with $k\in \N$ we let 
$$
F_i:=F_iS_n:=\{x\in \SS^n\ | x\equiv 1\mod (S^{in})\}
$$ 
Then we get a decreasing filtration 
\begin{equation}\label{filt}
\SS_n=F_0\supset F_{\frac{1}{n}}\supset F_{\frac{2}{n}}\supset 
\end{equation}
by normal subgroups and $\SS_n$ is complete and seperated with respect to this filtration, i.e. 
the canonical map $\SS_n\to \lim_i\SS_n/F_iS_n$ is an isomorphism. In particular $\SS_n$ 
is a {\it profinite group}.  Furthermore $F_{\frac{1}{n}}S_n$ is the kernel of the reduction homomorphism 
$$
\SS_n=\cO_n^{\times}\to \FF_{p^n}^{\times}\ . 
$$
This group is also  denoted by $S_n$ and is often called the {\it strict Morava stabilizer group}. 
\smallskip 
Furthermore for each $i={\frac{k}{n}}>0$ there are canonical isomorphism 
\begin{equation}\label{gr-iso}
F_{\frac{k}{n}}/F_{\frac{k+1}{n}}\to \FF_{p^n},\ \  x=1+aS^{k}\mapsto \overline{a}
\end{equation}
if $a\in \cO_n$ and if $\overline{a}$ denotes the residue class of $a$ in $\cO_n/(S)\cong \FF_{p^n}$. 
\smallskip
In particular $S_n/F_i$ is a finite $p$-group for each $i>0$ 
and $S_n$ is a {\it profinite $p$-group}. As $S_n$ is also normal in $\SS_n$, $S_n$ is the 
$p$-Sylow subgroup of the profinite group $\SS_n$. Furthermore the exact sequence 
$$
1\to S_n\to \SS_n\to \FF_{p^n}^{\times}\to 1 
$$
splits, i.e. $\SS_n\cong S_n\rtimes \FF_{p^n}^{\times}$ is a semidirect product. 
In fact, the splitting is given by Remark c on Witt vectors above.

\subsection{The associated mixed Lie algebra of $S_n$} 

The associated graded object $gr S_n$ with respect to the above filtration with 
$$
gr_i S_n: =F_{\frac{k}{n}}S_n/F_{\frac{k+1}{n}}S_n 
$$ 
for $i={\frac{k}{n}}$ becomes a graded Lie algebra with Lie bracket $[\bar a, \bar b]$ induced by
the commutator $[x,y]:=xyx^{-1}y^{-1}$ in $S_n$. 
\smallskip
Furthermore, if we define a function $\vp$ from  $\{\frac{k}{n}| k=1,2,\ldots\}$ 
to itself by $\vp (i): =\min\lbrace i+1,pi\rbrace$ then the $p$-th power map 
on $S_n$ induces maps 
$$
P: gr_iS_n\longrightarrow gr_{\vp (i)}S_n
$$ 
which define on $grS_n$ the structure of a mixed Lie algebra in the sense of Lazard \cite{Lz}. 
If we identify the filtration quotients with $\FF_{p^n}$ as above then the Lie bracket 
and the map $P$ are explicitly given as follows. 
\smallskip

\begin{prop}\label{gradlie}\cite{HennDuke}
Let $\bar a\in gr_iS_n$, $\bar b\in gr_jS_n$. With respect to the isomorphism (\ref{gr-iso}) 
the mixed Lie algebra structure maps are given by  

a) $$
[\bar a, \bar b]=
\bar a {\bar b}^{p^{ni}} - \bar b {\bar a}^{p^{nj}}\in gr_{i+j}S_n 
$$

b)
$$
P\bar a=\begin{cases}      
{\bar a}^{\frac{p^{pni}-1}{p^{ni}-1}}                &\  i<(p-1)^{-1} \\ 
{\bar a}+{\bar a}^{\frac{p^{pni}-1}{p^{ni}-1}}  &\  i=(p-1)^{-1} \\
{\bar a}                                                           &\  i>(p-1)^{-1}\  . \ \
\end{cases}
$$ 
\end{prop}
\smallskip

\begin{proof} a) Write $i={\frac{k}{n}}$, $j={\frac{l}{n}}$ and 
choose representatives $x=1+aS^{k}\in F_iS_n$, $y=1+bS^{l}\in F_jS_n$. Then
$x^{-1}=1-aS^{k} \mod S^{k+1}$, $y^{-1}=1-bS^{l} \ \mod S^{l+1}$  
and the formula 
$$
xyx^{-1}y^{-1}=1+((x-1)(y-1)-(y-1)(x-1))x^{-1}y^{-1}
$$
shows  
$$
xyx^{-1}y^{-1}=1+(aS^{k}bS^{l}-bS^{l}aS^{k}) \mod S^{k+l+1} \ .
$$
Because $\cO_n/(S)\cong \W(\FF_{p^n}) /(p)$ we can choose $a$ and $b$ 
from $\W(\FF_{p^n})$. Then $Sw=w^{\sigma}S$ and $w^{\sigma}\equiv w^p \mod (p)$ 
give the stated formula. 

\noindent
b) Again we write $i=\frac{k}{n}$ and we choose a representative $x=1+aS^{k}$ 
with $a\in \W_{\FF_{p^n}}$. Consider the expression $x^p=\sum_r \binom{p}{r}(aS^k)^r$. 
Because $\binom{p}{r}$ is divisible by $p$ for $0<r<p$ and because $S^n=p$ 
we get  
$$
x^p\equiv 1+aS^{n+k}+\ldots +(aS^k)^p \ \mod S^{2k+n} \ .
$$
Furthermore, modulo $S^{kp+1}$ we get  
$$
(aS^k)^p=a a^{\sigma^k} \ldots a^{\sigma^{(p-1)k}}S^{pk}
\equiv a a^{p^{k}}\ldots a^{p^{(p-1)k}}S^{pk}
\equiv a^{1+p^{k}+\ldots +p^{(p-1)k}}S^{pk} \ .  
$$
Now we only have to determine whether $pk$ is smaller resp. equal resp. 
larger than $n+k$. i.e. whether $pi$ is smaller resp. equal resp. larger 
than $1+i$. These cases are equivalent to $i<(p-1)^{-1}$ resp. 
$i=(p-1)^{-1}$ resp. $i>(p-1)^{-1}$ and hence we are done.  
\end{proof}

\subsection{Torsion in the Morava stabilizer groups}\label{tor}

As an immediate consequence of Proposition \ref{gradlie} we obtain the following result. 

\begin{cor}\label{tors} {\ }

\noindent
a) If $g\in F_i$ has finite order and $i>(p-1)^{-1}$  then $g=1$. 

\noindent
b) $S_n$ is torsionfree if $n$ is not divisible by $p-1$. \qed 
\end{cor}

\noindent
\underbar{Examples} 
a) In particular, if $n=1$ and $p>2$ and $n=2$ and $p>3$ then the groups $S_n$ are torsionfree. 

\noindent
b) For $n=1$ we have $\cO_n=\Z_p$, $S_1=\{x\in \Z_p^{\times}\ |\ x\equiv 1\mod (p)\}$. 
Furthermore, it is well known that  
$$
\Z_p^{\times}\cong 
\begin{cases} 
F_1\times \FF_{p^n}^{\times} & p>2 \\
F_2\times \{\pm 1\}               & p=2 \\
\end{cases}  
$$ 
and  $F_1$ is isomorphic to the additive group $\Z_p$ if $p$ is odd.  For $p=2$ it is $F_2$ which is 
isomorphic to the additive group $\Z_2$. 

\noindent 
c) For $n=2$ the group $S_2$ is nonabelian and its structure is complicated. 
Non-trivial torsion elements can exist only if $p=2$ or $p=3$. 
 
For $p=3$ a non-trivial torsion element must be nontrivial in $F_{\frac{1}{2}}/F_1$. 
An easy calculation shows that if $\omega$ is a fixed chosen primitive $8$-th root of unity in 
$\W_{\FF_9}$ then the element   
\begin{equation}\label{def-a}
a=-\frac{1}{2}(1+\omega S)
\end{equation} 
satisfies $a^3=1$. (It is clearly in $F_{\frac{1}{2}}$ and its image in $F_{\frac{1}{2}}/F_1$ is 
$\overline{\omega}$.)  

For $p=2$ there is always, i.e. for each $n$, the element $-1=1-S^n$ which is in $F_1$ 
and is a nontrivial element of order $2$. If $n=2$ there are elements of order $4$ 
which must be nontrivial in $F_{\frac{1}{2}}/F_1$. 

\noindent 
d) If $n=4$ and $p=2$ there is a chance for the existence of elements of order $8$ which are 
nontrivial in $F_{\frac{1}{4}}/F_{\frac{2}{4}}$. In fact, such elements exist and 
they are in the background of the recent solution of the  Kervaire invariant one problem by 
Hill, Hopkins and Ravenel.

\section{On the cohomology of the stabilizer groups with trivial coefficients}\label{triv-coh}
\medskip

The stabilizer groups are examples of $p$-adic Lie groups. For such groups the category of profinite 
modules over $\Z_p[[G]]$ has enough projectives and one can define continous cohomology with 
coefficients in a profinite $\Z_p[[G]]$-module $M$ simply as 
$H^s_{cts}(G,M)=\Ext^s_{\Z_p[[G]]}(\Z_p,M)$. Likewise one can define continuous homology with 
coefficients in a profinite $\Z_p[[G]]$-module $N$ simply as 
$H_s^{cts}(G,M)=\Tor^{\Z_p[[G]]}_s(\Z_p,N)$. In the sequel cohomology resp homology will always be 
continuous cohomology resp continuous homology and we will simply write is as $H^*(G,M)$ resp. 
$H_*(G,M)$. 

\subsection{$H_1$. The stabilizer group made abelian}
\medskip 

The commutator formula in Proposition \ref{gradlie} 
can be used to calculate the abelianization of the groups $S_n$. 
In this profinite setting it is the quotient 
$S_n/\overline{[S_n,S_n]}$ which identifies with the homology
$H_1(S_n,\Z_p)$.  (Here $\overline{E}$ denotes the closure of a given subset $E\subset S_n$). 
Likewise $H_1(S_n,\Z/p)$ identifies 
with the quotient $S_n/\overline{\langle [S_n,S_n],S_n^p\rangle}$.  
\medskip

Here is the crucial lemma on commutators.

\begin{lem}\label{commutator}  Let $p$ be any prime and let $k$ and $l$ be integers $>0$. 

\noindent
a) If $\frac{k+1}{n}$ is not an integer then the commutator map 
$gr_{\frac{k}{n}}S_n\otimes gr_{\frac{1}{n}}S_n\to gr_{\frac{k+1}{n}}S_n$ is onto.   

\noindent
b)  If $\frac{k+1}{n}$ is an integer then the image of the commutator map 
$gr_{\frac{k}{n}}S_n\otimes gr_{\frac{1}{n}}S_n\to gr_{\frac{k+1}{n}}S_n$ is equal 
to the kernel of the trace $tr:\FF_{p^n}\to \FF_p$

\noindent
c) If ${\frac{k+l}{n}}$ is an integer then the image of the commutator map 
$gr_{\frac{k}{n}}S_n\otimes gr_{\frac{l}{n}}S_n\to gr_{\frac{k+l}{n}}S_n$ is contained in the kernel of the 
trace $tr:\FF_{p^n}\to \FF_p$.
\end{lem} 

\begin{proof} a) By Proposition \ref{gradlie} the commutator map is given by the formula 
$$
[\bar a, \bar b]=
\bar a {\bar b}^{p^k} - \bar b {\bar a}^{p}
$$ 
By taking $b=1$ one sees that all elements of the form $\bar a -{\bar a}^p$ belong to the image. 
This is an $\FF_p$-linear subspace of $\FF_{p^n}$ of $\FF_p$-codimension $1$ 
which is contained in and therefore equal to the kernel of the trace.  
Furthermore, if $\frac{k+l}{n}$ is not an integer, 
it is enough to exhibit a couple $(\bar a ,\bar b)$ 
such that 
$$
tr(\bar a {\bar b}^{p^k} - \bar b {\bar a}^{p})=tr({\bar a}^p({\bar b}^{p^{k+1}} - \bar b)\neq 0 \ . 
$$ 
Now, if $k+1$ is not divisible by $n$ there exists $\bar b$ such that 
$c:={\bar b}^{p^{k+1}} - \bar b\neq 0$. Because the trace is a nontrivial linear form and because  
$$
\FF_{p^n}\to \FF_{p^n}, \ \bar a\to {\bar a}^pc
$$ 
is bijective we are done. 

\noindent
b) If $\frac{k+1}{n}$ is an integer, i.e. $k+1$ divisible by $n$, then  ${\bar b}^{p^{k+1}} - \bar b=0$ 
for all $\bar b$ and therefore 
$$
tr(\bar a {\bar b}^{p^k} - \bar b {\bar a}^{p})=tr({\bar a}^p({\bar b}^{p^{k+1}} - \bar b)=0 \ . 
$$ 
On the other hand we have already seen in the proof of (a) that the kernel of the trace is in the image of the commutator map. 

\noindent
c) In general the commutator map 
$gr_{\frac{k}{n}}S_n\otimes gr_{\frac{l}{n}}\SS_n\to gr_{\frac{k+l}{n}}\SS_n$ 
is given by 
$$
[\bar a, \bar b]=
\bar a {\bar b}^{p^k} - \bar b {\bar a}^{p^l}
$$ 
and hence 
$$
tr(\bar a {\bar b}^{p^k} - \bar b {\bar a}^{p^l})=tr({\bar a}^{p^l}({\bar b}^{p^{k+l}} - \bar b))\ .    
$$ 
If $\frac{k+l}{n}$ is an integer then $k+l$ is divisible by $n$ and hence 
${\bar b}^{p^{k+l}} - \bar b=0$. 
\end{proof}
\smallskip

\begin{prop} Let $p$ be an odd prime and $n>1$. Then 
$$
H_1(S_n,\Z_p)\cong \Z_p\oplus (\Z/p)^{n}\ . 
$$ 
As topological generator of $\Z_p$ one can choose $1+cS^n=1+2c$ where 
$c$ is in $\W_{\FF_{p^n}}$ of valuation $0$ with $tr(\bar c)\neq 0$ and as generators 
of the $n$ summands $\Z/p$ one can choose the elements 
$1+\omega^{p^j}S$, $j=0,\ldots,n-1$ of order $p$ where $\omega$ 
is a fixed primitive root of unity of order $p^n-1$.   
\end{prop} 

\begin{proof}  The filtration on $S_n$ introduced in (\ref{filt}) 
induces one on $S_n/\overline{[S_n,S_n]}$ and Lemma \ref{commutator} shows that 
$gr_i(S_n/\overline{[S_n,S_n]})$ is isomorphic to $gr_iS_n\cong \FF_q$ if $i=\frac{1}{n}$, isomorphic 
to the image of $tr:\FF_q\to \FF_p)$ if $i$ is an integer, and zero otherwise.  
By Proposition \ref{gradlie} the induced $p$-th power map 
sends $gr_i(S_n/\overline{[S_n,S_n]})$ isomorphically to $gr_{i+1}(S_n/\overline{[S_n,S_n]})$ 
if $i$ is an integer, and it is clearly trivial on 
$gr_{\frac{1}{n}}(S_n/\overline{[S_n,S_n]})$ except possibly if $n=p$.   
Furthermore, if $n=p$ we get $tr(P(\bar a))=tr({\bar a}^{1+p+\ldots p^{(p-1)}})=0$ 
because ${\bar a}^{1+p+\ldots p^{(p-1)}}$ is fixed by Frobenius 
and thus the trace is $p$ times this element, hence trivial modulo $p$. 
Now Lemma \ref{commutator} implies that the induced $p$-th power map is always 
trivial on $gr_{\frac{1}{n}}(S_n/\overline{[S_n,S_n]})$ and this implies the result.  
\end{proof}

\begin{prop} Let $p=2$ and $n>1$. Then 
$$
H_1(S_n,\Z_2)\cong \Z_2\oplus (\Z/2)^{n+1}\ . 
$$ 
As topological generator of $\Z_2$ one can choose  $1+cS^{2n}=1+4c$ 
and as generators of the $n+1$ summands $\Z/2$ one can choose the elements 
$1+cS^n$, $1+\omega^{2^k}S$, $k=1,\ldots,n$,  
where $c$ is in $\W_{\FF_{2^n}}$ of valuation $0$ and $tr(\bar c)\neq 0$,  
and $\omega$ is a fixed primitive root of unity of order $2^n-1$.  
\end{prop}

\begin{proof}  Again the filtration on $S_n$ introduced in (\ref{filt}) 
induces one on $S_n/\overline{[S_n,S_n]}$ and the previous lemma shows that 
$gr_i(S_n/\overline{[S_n,S_n]})$ is isomorphic to $gr_iS_n\cong \FF_q$ if $i={\frac{1}{n}}$ and 
isomorphic the image of $tr:\FF_q\to \FF_p$ if $i$ is an integer, and zero otherwise.  
By Proposition \ref{gradlie}  the induced $p$-th power map on $gr_i(S_n/\overline{[S_n,S_n]})$ 
sends $gr_i(S_n/\overline{[S_n,S_n]})$ isomorphically to $gr_{i+1}(S_n/\overline{[S_n,S_n]})$ 
if $i$ is an integer $>1$, and it is clearly trivial on 
$gr_i$ except possibly if $i=\frac{1}{2}$ or $i=1$. The same argument as in the previous proof 
shows that  the induced $p$-th power map is  trivial on $gr_{\frac{1}{2}}$. 
For $i=1$ Proposition \ref{gradlie} gives 
$$
P(\bar a)=\bar a +\bar a^{\frac{2^{2n-1}}{2^n-1}}=\bar a +\bar a^{2^n+1}=\bar a +\bar a^{2}\ . 
$$ 
The trace of this element is again trivial and the result follows once 
again by Lemma \ref{commutator}. 
\end{proof}
\smallskip

\begin{cor}\label{H1} Let $p$ be a prime and $n>1$. Then 
$$
H_1(S_n,\Z/p)\cong 
\begin{cases} 
(\Z/p)^{n+1}    & p>2 \\
(\Z/2)^{n+2} & p=2 \ . \qed  \\
\end{cases} 
$$
\end{cor} 
\smallskip

\subsection{The chomology of  $S_1$} 
\medskip

This case is fairly easy. 

\begin{prop}\label{cohS1} {\ }

\noindent
a) If $p$ is odd then 
$$
H^*(S_1,\Z_p)\cong 
\begin{cases} 
\Z_p & n=0,1 \\
0      & \text{else} \ . \\
\end{cases} 
$$
b) If $p=2$ then 
$$
H^*(S_1,\Z_p)\cong 
\begin{cases} 
\Z_2 &  n=0,1 \\
\Z/2  &  n\geq 2\ . \\ 
\end{cases}  
$$
\end{prop} 

\begin{proof} We have $S_1=\Z_p$ if $p>2$ and $S_1=\Z_2^{\times}\cong \Z/2\times \Z_2$ if $p=2$. 
The result follows therefore as soon as we know that $H^n(\Z_p,\Z_p)\cong \Z_p$ 
if $n=0,1$ and trivial otherwise. (For $p=2$ we use the Kuenneth theorem). 
\smallskip
Now cohomology is calculated from a resolution of the trivial module $\Z_p$ by projective modules over 
the completed group ring $\Z_p[[\Z_p]]$. There is an obvious algebra homomorphism 
from the polynomial algebra $\Z_p[T]$ to the group algebra $\Z_p[\Z_p]$ which sends $T$ to $t-e$ 
where $t$ is a topological generator of the  group $\Z_p$. 
This map extends to a continuous homomorphism from the power series ring 
$\Z_p[[T]]$ to $\Z_p[[\Z_p]]$ which can be checked to be an isomorphism. 
In fact, this isomorphism is the starting point for Iwasawa theory in number theory (cf. \cite{NSW}). 
Now it is obvious that the trivial $\Z_p[[T]]$-module $\Z_p$ admits a 
projective resolution 
\begin{equation}\label{Iwasawa-res}
0\to \Z_p[[T]]\buildrel{T}\over\longrightarrow  \Z_p[[T]]\to \Z_p
\end{equation}
and the result follows.  
\end{proof}

\subsection{Structural properties of $H^*(S_n,\Z/p)$}\label{struct-prop}
\medskip

Proposition \ref{cohS1} and its proof yield immediately the additive structure of $H^*(S_1,\Z/p)$ resp. of 
$H^*(\Z_p,\Z_p)$. In fact, there is a cup product structure which is uniquely determined 
by the additive result.  
\smallskip

\begin{prop} Let $p$ be any prime. Then 
$$
H^*(\Z_p,\Z/p)\cong \Lambda_{\Z/p}(H^1(S_1,\Z/p))\cong \Lambda(e)
$$ 
with $e\in H^1$ given by the reduction homomorphism $\Z_p\to \Z/p$ considered as an element 
in $H^1(\Z_p,\Z/p)=\Hom(\Z_p,\Z/p)$ and $\Lambda_{\Z/p}$ denotes the exterior algebra over 
$\Z/p$. \qed 
\end{prop}

\noindent
Via the Kuenneth theorem we get the following corollary. 

\begin{cor} Let $p$ be any prime. Then 
$$
H^*(\Z_p^n,\Z/p)\cong \Lambda_{\Z/p}(H^1(\Z_p^n,\Z/p))\cong \Lambda_{\Z/p}(e_1,\ldots ,e_n)\ . 
$$
with $e_i\in H^1$ for $i=1,\ldots,i=n$, a dual basis of $\Z_p^n/(p)$ and $\Lambda_{\Z/p}$ denoting the 
exterior algebra over $\Z/p$. 
\qed 
\end{cor} 

An interesting feature of the stabilizer groups is that although they do not contain 
abelian subgroups of rank $>n$ (i.e. free $\Z_p$ submodules of rank $>n$) 
they do contain finite index subgroups which look abelian of rank $n^2$ 
from the point of view of mod-$p$ cohomology. The following result follows from \cite{Lz}. 

\begin{prop}\label{AbelianCoh} {\ } 

\noindent
a) Let $p>2$  and let $i=\frac{k}{n}\geq 1$. Then 
$$
H^*(F_i,\Z/p)\cong \Lambda_{\Z/p}(H^1(F_i,\Z/p))\cong \Lambda_{\Z/p}(e_{i,j}) 
$$ 
where $0\leq i,j\leq n-1$ and $e_{i,j}$ is dual to to $1+\omega^iS^{k+j}$. 

\noindent
b) For $p=2$ the same result holds if $i=\frac{k}{n}> 1$. 
\end{prop} 
\medskip 

\begin{cor} The mod-$p$ cohomology ring of $S_n$ is a noetherian 
algebra over $\Z/p$. 
\end{cor}

\begin{proof} This follows from Proposition \ref{AbelianCoh} 
by analyzing the spectral sequence of the group 
extension $1\to F_i\to S_n\to S_n/F_i\to 1$.  
\end{proof}
\medskip

\begin{defn} Let $p$ be any prime. A profinite $p$-group is called a 
{\it Poincar\'e duality group of dimension $d$} if 

$\bullet$ $H^s(G,\Z/p)$ is finite dimensional for each $s\geq 0$ 

$\bullet$ $H^d(G,\Z/p)\cong \Z/p$ 

$\bullet$ The cup product $H^s(G,\Z/p)\times H^{d-s}(G,\Z/p)\to H^d(G,\Z/p)$ is a nondegenerate 
bilinear form for each $s>0$. 
\end{defn} 
\medskip

\noindent
\underbar{Examples} a) $\Z_p^d$ is a Poincar\'e duality group of dimension $d$. 

\noindent 
b) $F_iS_n$ is a Poincar\'e duality group of dimension $n^2$ whenever $i=\frac{k}{n}\geq 1$ if $p>2$,  
and whenever $i>1$ if $p=2$. 
\medskip

\begin{thm}\label{STV}\cite{Lz} Suppose that $G$ is a profinite $p$-group without 
torsion which contains a finite index subgroup which is a Poincar\'e duality group of dimension $n$.  
Then $G$ is itself a Poincar\'e duality group of dimension $n$. 
\end{thm}

\subsection{The reduced norm and a decomposition of $S_n$}\label{reduced}   

If $n=2$ and $p>3$ then $S_2$ is torsionfree and hence it is a Poincar\'e duality group of 
dimension $2^2=4$. In fact, we can even reduce to the case of a Poincar\'e duality group $3$ 
as follows. 

In the case of general $n$ and $p$ we consider  $\cO_n$ as a left $\W(\FF_{p^n})$-module of rank 
$n$. Multiplying on the right gives a multiplicative homomorphism 
$$
\cO_n\to M_n(\W(\FF_{p^n}))
$$ 
and hence 
$$
\SS_n\to GL_n(\W(\FF_{p^n}))\ . 
$$
Following this by the determinant gives a homomorphism $\SS_n\to (\W(\FF_{p^n}))^{\times}$ which is 
invariant with respect to the natural actions of $Gal(\FF_{p^n}:\FF_p)$. 
On the other hand we have noted in the remark preceeding Definition \ref{stab} that 
the Galois action on $\SS_n$ is induced by conjugation 
by the element $S$ in $D_n^{\times}$.  
It follows that the determinant restricted to $\SS_n$ takes its values in the 
Galois invariant part $\Z_p^{\times}$ of $\W(\FF_{p^2})^{\times}$. 

If $n=2$ this can also be seen by an easy calculation as follows. If we choose $1$ and $S$ as 
$\W(\FF_{p^2})$ basis for $\cO_2$ then right multiplication defines 
$$
\cO_2\to M_2(\W(\FF_{p^2})),\  a+bS\mapsto \begin{pmatrix} a & pb^{\sigma}\\ b & a^{\sigma}
\end{pmatrix} 
$$ 
with obviously Galois-invariant determinant.   

The resulting homomorphism $\SS_n\to \Z_p^{\times}$ is often called the {\it reduced norm}. 
Restricted to the central $\Z_p^{\times}$ in $\SS_n$ the reduced norm is given by the $n$-th power  
map. By restricting to the $p$-Sylow subgroup and assuming that $p$ does not divide $n$ 
we get a splitting of the sequence 
$$
1\to S_n^1\to S_n\to P(\Z_p^{\times})\to 1 
$$ 
where $P(\Z_p^{\times})$ is the $p$-Sylow subgroup of $\Z_p^{\times}$.

\smallskip

\begin{prop} Suppose $p$ does not divide $n$. Then the group 
$S_n$ is isomorphic to the direct product of its subgroups $S_n^1$ and $P(\Z_p^{\times})$, i.e. 
$$
S_n\cong 
\begin{cases} 
S_n^1\times \{g\in \Z_p^{\times}\ |\ g\equiv 1\mod (p)\} & p>2  \\
S_n^1\times \Z_2^{\times} & p=2 \ .\ \ \ \qed\\
\end{cases} 
$$
\end{prop} 
\medskip

\subsection{Cohomology in case $n=2$ and $p>2$}\label{coh-n=2-p>3}

\subsubsection{The case $p>3$}   
In this case we have 
$$
S_2\cong S_2^1\times \{g\in \Z_p^{\times}\ |\ g\equiv 1\mod (p)\}\cong S_2^1\times \Z_p\ 
$$ 
The group $S_2$ is a Poincar\'e duality group of dimension $4$, hence $S_2^1$ 
is a Poincar\'e duality group of dimension $3$. Calculating its mod $p$-cohomology 
is therefore easy. By Poincar\'e duality it is enough to calculate $H^1(S_2^1,\FF_p)$. From 
Corollary \ref{H1}  we obtain the following result. 
\smallskip

\begin{thm}\cite{HennRes}  Let $p>3$. Then 
$$
H^*(S_2^1,\Z/p)\cong 
\begin{cases} 
\Z/p & *=0,3 \\
(\Z/p)^2 &* =1,2\\ 
0      & *>3 \ \ \ . \ \ \qed  \\
\end{cases}  
$$
\end{thm} 

\subsubsection{The case $p=3$} 

The cases $n=2$ and $p=2,3$ are considerably more complicated. In this case the groups $S_n$ 
do contain $p$-torsion and they are no longer Poincar\'e duality groups. In fact, their $vcd_p$ is infinite. 
We will be content to discuss the case $p=3$. 
\smallskip
For $p=3$ we still have the decomposition 
$$
S_2\cong S_2^1\times \{g\in \Z_p^{\times}\ |\ g\equiv 1\mod (p)\}\cong S_2^1\times \Z_p
$$ 
and the problem is again reduced to the case of $S_2^1$. Even though the group $S_2^1$ 
is not a Poincar\'e duality group it contains one of 
index 9, namely the group $F_1S_2^1=F_1S_2\cap S_2^1$. In fact, it even contains one of index $3$. 
In order to see this we consider the formula for the $3$-rd power map 
$$
P:F_{\frac{1}{2}}\to F_1,\  \bar a \mapsto \bar a +{\bar a}^{1+3+9}\ . 
$$ 
This shows that if if there is an element $g\in S_2^1$ of order $3$ then it has the form 
$g=1+aS\mod F_1$ with $\bar a ^4=-1$.  Thus if we define $K$ to be the kernel of the homomorphism  
$S_2^1\to S_2^1/F_1S_2^1\cong \FF_9\to \FF_9/\FF_3$ 
then $K$ is torsion-free and by Theorem \ref{STV} it is a Poincar\'e duality group of dimension $3$.

\begin{prop}\cite{HennDuke}
$$
H^*(K,\Z/3)\cong 
\begin{cases} 
\Z/3  & *=0,3 \\
(\Z/3)^2 & *=1,2 \\
0           & *>3 \ .\\
\end{cases} 
$$
\end{prop}

\begin{proof} Because $K$ is without torsion Theorem \ref{STV} implies that 
it is a Poincar\'e duality group (of dimension $3$). So it is enough to calculate 
$H^1(K,\Z/3)\cong H_1(K,\Z/3)$. For this we consider the filtration on $K$ given 
by $F_iK:=K\cap F_1S_2$.  It is not hard to check that 
$H_1(K,\Z_3)\cong \Z/9\oplus \Z/3$ generated by $b:=[a,\omega]$  and $c=[a,b]$ where as before 
$$
a=-\frac{1}{2}(1+\omega S)
$$ 
is the element of order $3$ of  (\ref{def-a})  and $\omega$ is a primitive $8$th root of unity in 
$\W(\FF_9)$. This implies the desired result. 
\end{proof}
\smallskip

The cohomology of $S_2^1$ can now be calculated by using the (non-central) split exact sequence 
$$
1\to K\to S_2^1\to \Z/3\to 1\ . 
$$ 
The quotient map $S_2^1\to \Z/3$ makes $H^*(S_2^1,\Z/3)$ into a module over the 
polynomial algebra generated by $y\in H^2(\Z/3,\Z/3)$. 
It is true (but far from obvious) that this spectral sequence degenerates at $E_2$. In fact, 
it is equivalent to knowing that $H^*(S_2^1,\Z/3)$ is a free module over the polynomial algebra 
$\Z/3[y]$. Using that we obtain the following result. 
\smallskip 

\begin{thm}\cite{HennDuke} Let $p=3$. 
Then $H^*(S_2^1,\Z/3)$ is a free module over $\Z/3[y]$ on $8$ generators in degrees  
$0,1,1,2,2,3,3,4$. 
\end{thm} 

The cup product structure is also known. It can be approached as follows. Up to conjugacy 
there are two subgroups of order $3$ in $S_2^1$, namely the subgroup $\langle a\rangle$ 
generated by $a$ and the subgroup $\langle \omega a\omega^{-1}\rangle$. 
The centralizers of these elements are  isomorphic and 
$C_{S_2^1}\langle a\rangle \cong \langle a\rangle\times \Z_3$. 
The cup product structure is detemined by the following result. 
\medskip

\begin{thm}\cite{HennDuke}
a) The restriction homomorphisms induce a monomorphism 
$$
H^*(S_2^1,\Z/3)\to H^*(C_{S_2^1}\langle a\rangle,\Z/3)\times 
H^*(C_{S_2^1}\langle \omega a\omega^{-1}\rangle,\Z/3)
$$
whose target is isomorphic to $\prod_{i=1}^2\Z/3[y_i]\otimes\Lambda_{\Z/3}(x_i,a_i)$ 
where the elements $y_i$ are of degree $2$ and $x_i$ and $a_i$ are of degree $1$.  

\noindent
b) This map is an isomorphism in degrees $>2$. Its image in degree $0$ is the diagonal, in degree $1$ 
it is the subspace generated by $x_1$ and $x_2$ and in degree $2$ the subspace generated by 
$y_1$, $y_2$ and $x_1a_1-x_2a_2$. 

\noindent
c) The image is a free module over $\Z/3[y_1+y_2]$ on the following $8$ generators: $1$, $x_1$, $x_2$, 
$y_1-y_2$, $x_1a_1-x_2a_2$, $y_1a_1$, $y_2a_2$, $y_1x_1a_1+y_2x_2a_2$. 
\end{thm}

\noindent
\subsubsection{The case $p=2$} The case of the prime $2$ is even more complicated 
but it is also understood (cf. \cite{HennRes} and the recent Northwestern 
theses of Beaudry and Bobkova). 
\medskip

\section{Cohomology with non-trivial coefficients and resolutions}

For homotopy theoretic applications we will be interested in calculating cohomology 
with certain non-trivial coefficients, in particular $H^*(\GG_n,(E_n)_*)$. 
For this we use explicit resolutions of the trivial module. 
In this section we will discuss the classical case $n=1$ in fair detail and briefly 
comment on the case $n=2$.

\subsection{The case $n=1$}\label{n=1}  
In the case $n=1$ we have already seen such resolutions for the group $S_1$, at least if $p>2$. 
More precisely we have seen in (\ref{Iwasawa-res}) seen that there is a free resolution 
$$
0\to P_1\buildrel{t-e}\over\longrightarrow P_0\to \Z_p\to 0 
$$ 
of the trivial $\Z_p[[\Z_p]]$-module with $P_0=P_1=\Z_p[[\Z_p]]$ and $t$ a topological generator 
of $\Z_p$. In the case of $\SS_1=\GG_1$ 
we can use the same resolutions but enriched as resolutions by $\Z_p[[\GG_1]]$-modules. 
In fact because of the product decomposition $\GG_1=\Z_p^{\times}\cong \Z_p\times F$ 
where $F=\{\pm 1\}$ if $p=2$ resp. $F={\Z/p}^{\times}$ if $p>2$,  
every $\Z_p[[\Z_p]]$-module resp. every $\Z_p[[\Z_p]]$-module homomorphism 
can be considered as a $\Z_p[[\GG_1]]$-module resp. $\Z_p[[\GG_1]]$-module homomorphism 
via the projection map $\GG_1\to \Z_p$. Of course, in this case the modules are no longer 
free modules. However, if $p$ is odd they are still projective and in case $p=2$ 
they are at least permutation modules. Indeed as $\Z_p[[\GG_1]]$-modules we always have 
$P_0=P_1=\Z_p[[\GG_1/F]]=\Z_p[[\GG_1]]\otimes_{\Z_p[F]}\Z_p$ 
and the trivial $\Z_p[F]$-module $\Z_p$ is projective if $p>2$ because the order of $F$ is prime to 
$p$ in this case. But even  in the case $p=2$ this permutation resolution  
is useful for calculating group cohomology. 
In fact, it gives rise to a long exact sequence (with $R=\Z_p[[\GG_1]]$ and $\psi$ 
denoting a topological generator of $\Z_p^{\times}/F\cong \Z_p$) 
$$
\ldots 
\to \Ext^s_R(\Z_2,M)\to \Ext^s_R(P_0,M)
\buildrel{\psi -id}\over\longrightarrow \Ext^s_R(P_1,M)\to \Ext^{s+1}_R(\Z_2,M)\to \ldots 
$$ 
which can be identified by definition of $H^*$ and by using Shapiro's Lemma with 
\begin{equation}\label{exact-seq}
\ldots \to H^s(\GG_1,M)\to H^s(F,M)\buildrel{\psi -id}\over\longrightarrow H^s(F,M)
\to H^{s+1}(\GG_1,M)\to \ldots \ . 
\end{equation}

\subsubsection{The case $p>2$}\label{p>2} 
\smallskip 
If $p>2$ the groups in the middle of (\ref{exact-seq}) are trivial unless $s=0$. 
Now we consider the graded module $M=(E_1)_*=\Z_p[u^{\pm 1}]$ with $|u|=-2$. The action of 
$\GG_1=\Z_p^{\times}$ on this graded algebra is by algebra homomorphisms 
and is thus specified by the action on the polynomial generator $u$. It is the tautological action 
$(g,u)\mapsto g.u$. Then we get 
$$
H^*(F,\Z_p[u^{\pm 1}])=
\begin{cases} 
\Z_p[u^{\pm (p-1)}] & s=0 \\
0 & s\neq 0\ . \\ 
\end{cases} 
$$ 
For $\psi$ we can take the element $p+1\in\Z_p^{\times}$. Then 
$$
(\psi-id)_*(u^{t(p-1)})=((p+1)^{t(p-1)}-1)u^{t(p-1)}=cp^{\nu_p(t)+1}u^{t(p-1)} 
$$ 
where $\nu_p(t)$ is the $p$-adic valuation of the integer $t$ and $c$ is a unit modulo $p$. 
This proves the following result. 
\smallskip

\begin{thm} Let $p$ be an odd prime. Then 
$$
H^s(\SS_1,(E_1)_t)=
\begin{cases} 
\Z_p  & t=0,\ s=0,1 \\ 
\Z/p^{\nu_p(t')+1} & t=2(p-1)t', s=1 \\
0                         & else \ . \ \ \ \ \qed \\ 
\end{cases} 
$$
\end{thm} 

\noindent
Then the homotopy fixed point spectral sequence (\ref{descentSS}) 
$$
E_2^{s,t}=H^s(\SS_1,(E_1)_t)\Longrightarrow \pi_{t-s}(L_{K(1)}S^0)
$$ 
collapses by sparseness, and we get 
the following result which is essentially equivalent to Ravenel's calculation of $\pi_*L_1S^0$ 
(cf. \cite{Ravloc}).  
\smallskip

\begin{thm} Let $p$ be an odd prime. Then  
$$
\pi_n(L_{K(1)}S^0)\cong 
\begin{cases} 
\Z_p & n=0,-1 \\
\Z/p^{\nu_p(t')+1} & n=2(p-1)t'-1 \ . \ \ \ \ \qed \\
\end{cases}
$$
\end{thm}

\subsubsection{The case $p=2$}\label{p=2}

For $p=2$ we get 
\begin{equation}\label{C2-descent}
H^s(C_2,\Z_2[u^{\pm 1}])=
\begin{cases} 
\Z_2[u^{\pm 2}]  & s=0 \\
\Z/2[u^{\pm 2}]\{y^{s'}\}  & s=2s' \\
\Z/2[u^{\pm 2}]\{y^{s'}x\}  & s=2s'+1\\
\end{cases} 
\end{equation} 
with $y\in H^2(C_2,\Z_2)$, $x\in H^1(C_2,\Z_2\{u\})$. 
We note that this group cohomology 
is the $E_2$-term for the homotopy fixed point spectral sequence  converging to 
$\pi_*KO\Z_2=\pi_*KU^{hC_2}$. The bidegree of $u$, $y$ and $x$ are $|u|=(0,-2)$, $|y|=(2,0)$ and 
$|x]=(1,-2)$. The full multiplicative stucture is determined by the relation $x^2=yu^2$. 
One can thus rewrite this $E_2$-term as 
\begin{equation}
E_2^{*,*}=\Z/2[u^{\pm 2},\eta]/(2\eta)
\end{equation} 
with $\eta=xu^{-2}$. The notation is chosen so as to agree with usual notation in homotopy theory, i.e. 
$\eta \in E_2^{1,2}$ is a permanent cycle which represents the image of 
$\eta\in \pi_1^S$ in $\pi_1(KO\Z_2)$. 
\smallskip
In the homotopy fixed point spectral sequence 
there is a single differential. In fact, in the sphere we have $\eta^4=0$ 
so $\eta^4$ has to be hit by a differential. There is only one way how this can happen, namely via 
\begin{equation}\label{d3-diff}
d_3(u^{-2})=\eta^3\ . 
\end{equation}
The spectral sequence is multiplicative. Therefore we get 
$$
E_4^{s,t}=\Z_2[u^{\pm 4}]\{2u^2,\eta, \eta^2\}/(2\eta) \ ,  
$$
the spectral sequence degenerates at $E_4$ and we find the well known 
homotopy groups of $\pi_*(KO\Z_2)$ given as 
$$
\pi_s(KO\Z_2)\cong 
\begin{cases} 
\Z_2 & s\equiv 0,4\mod(8)\\ 
\Z/2  & s\equiv 1,2 \mod(8)  \\
0      & s\equiv 3,5,6,7 \mod(8)\ . \\
\end{cases} 
$$ 
Independently of this homotopy theoretic calculation we can use (\ref{C2-descent}) 
and the long exact sequence 
(\ref{exact-seq}) in order to calculate $H^s(\GG_1,(E_1)_*)$. 
The induced homomorphism in $H^s$ is trivial 
if $s>0$ because $\psi$ necessarily acts trivially on $H^s(C_2,(E_1)_t)$ if $s>0$ 
(because the group is either trivial or $\Z/2$). 
For $s=0$ we have a similar phenomenon as before, namely we take for $\psi$ the element 
$3=1+2\in \Z_2^{\times}$.  
Because of 
$$
(\psi-id)_*(u^{2t})=((1+2)^{2t}-1)u^{2t}=((1+8)^t-1)u^{2t}
=c2^{\nu_2(t)+3}u^{2t} 
$$ 
with $c$ a unit modulo $2$ we obtain the following result. 

\smallskip

\begin{thm} Let $p=2$. Then 
$$
H^s(\SS_1,(E_1)_t)\cong 
\begin{cases} 
\Z_2                      & t=0, s=0,1 \\
\Z/(2^{\nu_2(t')+3}) & s=1, t=4t'\neq 0 \\
\Z/2                       & s=1, t=4t'+2  \\
\Z/2                       & s\geq 2, t\ even \\
0                           & else \ . \ \ \ \qed \\
\end{cases}  
$$
\end{thm}  
\smallskip

In cohomological degrees $s\geq 2$ the $E_2$-term of the 
homotopy fixed point spectral sequence (\ref{descentSS}) 
$$
E_2^{s,t}=H^s(\SS_1,(E_1)_t)\Longrightarrow \pi_{t-s}(L_{K(1)}S^0)
$$ 
agrees with the algebra $\Z_2[u^{\pm 2},\eta,\zeta]/(2\eta,\zeta^2)$ 
with $\zeta\in H^1(\SS_1,\Z_2[u^{\pm1}]_0)$. In cohomological degree $s=0$ the $E_2$-term is 
isomorphic to $\Z_2$ concentrated in internal degree $t=0$. In bidegrees $(1,4t'+2)$ 
it is isomorphic to $\Z/2$ generated by $\eta u^{-2t'}$ 
and in bidegrees $(1,4t')$ 
it is isomorphic to  $\Z/(2^{\nu_2(t')+3})$ resp. $\Z_2$ generated by $\zeta u^{-2t'}$ if $t'\neq0$ 
resp. by $\zeta$ if $t'=0$. 

To get at the homotopy of $L_{K(1)}S^0$ we still need to understand the differentials in this spectral 
sequence. They are determined via naturality and via the geometric boundary theorem \cite{Br} 
by those for the homotopy fixed point 
spectral sequence for $KO\Z_2$, i.e. by  (\ref{d3-diff}). 
The $d_3$-differential is linear with respect to the permanent 
cycles $\eta$ and $\zeta$,  and it is determined by 
$$
d_3(\eta u^{-2t})=
\begin{cases}
\eta^4u^{-2t+2} & t\equiv 1\mod(2) \\ 
0                       & t\equiv 0\mod(2)\\ 
\end{cases} 
$$
and 
$$
d_3(\zeta u^{-2t})=
\begin{cases} 
\zeta\eta^3u^{-2t+2} & t\equiv 1\mod(2) \\
0                       & t\equiv 0\mod(2)  \ . \\
\end{cases} 
$$
By sparseness all higher differentials are trivial.     
In dimension $\equiv 1,3\mod(8)$ there are still extension problems to be solved. 
This can be done, for example, by comparing with the calculation for the mod-$2$ Moore space.  
In dimension $\equiv 1\mod(8)$ the extensions turn out to be trivial while in dimensions 
$\equiv 3\mod(8)$ they are non-trivial. The final result reads as follows and is essentially once again 
equivalent to Ravenel's calculation of $\pi_*L_1S^0$ in  \cite{Ravloc}.  

\smallskip

\begin{thm} Let $p=2$. Then 
$$
\pi_n(L_{K(1)}S^0)\cong
\begin{cases} 
\Z_2\oplus\Z/2   & \ i=\ 0   \\
\Z/2             & \ i\equiv\ 0\ \mod(8), i\neq 0 \\
\Z/2\oplus \Z/2  & \ i\equiv\ 1\ \mod(8) \\
\Z/2              & \ i\equiv\ 2\ \mod(8)\\
\Z/8              & \ i\equiv\ 3\ \mod(8) \\
\Z_2             & \ i=-1   \\
\Z/(2^{\nu_2(s)+4})     & \ i=8s-1, i\neq -1 \\
0                  & \ \text{otherwise} \ . \ \ \ \ \qed \\
\end{cases}
$$
\end{thm}
\medskip

\subsection{Some comments on the case $n=2$} 
\medskip 

As a first step one needs to establish the resolutions of the trivial module described in sections 
\ref{n=2-p>3} and \ref{n=2-p=3} and make them into effective calculational tools.  

For $p>3$ these resolutions are projective minimal resolutions 
which are constructed from the calculations of 
$H^*(S_2^1,\Z/p)$ discussed in section \ref{coh-n=2-p>3}. 
For example, one can construct a resolution of the trivial 
$S_2^1$-module $\Z_p$ of the form 
$$
0\to P_3\to P_2\to P_1\to P_0\to \Z_p\to 0\  
$$
such that $P_0=P_3=\Z_p[[S_2^1]]$ and $P_1=P_2=\Z_p[[S_2^1]]^{\oplus 2}$.  
As in the case $n=1$ this projective resolution can be promoted to one of 
$\GG_2^1$. However, it is not true that $S_2^1$ is a quotient of $\GG_2^1$ 
so the details of promoting this resolution to one of $\GG_2^1$ are not as straightforward. 

Nevertheless the existence of the resolution is fairly formal and knowledge of the existence of such 
resolutions is already quite useful. However, mere existence is not good enough to carry out actual 
calculations like that of $H^*(\GG_2^1,(E_2)_*)$. 
A great deal of work is necessary in order to describe the homomorphism 
in the resolution explicitly, or at least closely enough such that actual calculations 
can be carried out.  Another problem is that of getting sufficient 
control of the action of $\GG_2$ on $(E_2)_*$. 
At least modulo $p$ all these problems have been resolved in the thesis of O. Lader 
\cite{Ld} (cf. section \ref{p>3}). 

In the case of the prime $3$ the resolution has a similar form except that the modules are no 
longer projective (just as in the case $p=2$ and $n=1$ before). Nevertheless in the thesis of Karamanov 
\cite{kar} and in \cite{HKM} this resolution has been made into a very effective computational tool 
(cf. section \ref{p>3}).

As noted in section \ref{p=2a} the case of the prime $2$ is currently actively developped.

\bibliographystyle{amsplain}

\bibliographystyle{amsplain}
\bibliography{bibghm}
\bigbreak
\bigbreak

\end{document}